\documentclass[a4paper,10pt]{amsart}

\usepackage{amsmath}
\usepackage{amssymb}
\usepackage{graphicx}
\usepackage{xcolor}

\usepackage{amsthm}
\theoremstyle{definition}
\newtheorem{proposition}{Proposition}[section]
\newtheorem{theorem}{Theorem}[section]

\newtheorem{lemma}{Lemma}[section]

\newtheorem{remark}{Remark}[section]

\theoremstyle{remark}

\begin{document}

\title[Orthogonal positive sums in $C^*$-algebras]{On the set of orthogonal positive sums in a $C^*$-algebra}

\author{Ali Taghavi}
\address{Department of  Mathematics, Qom University of Technology, Qom, Iran.}
\email{taghavi@qut.ac.ir}

\date{\today}



\begin{abstract}
For a $C^*$-algebra $A$ let $\mathcal I(A)$ be the set of all elements
$a+b$ with $a,b\in A_+$, $ab=0$ and $\|a\|=\|b\|=1$. We study the
algebraic, metric and topological properties of $\mathcal I(A)$. We
describe $\mathcal I(A\oplus B)$ and show that
$\mathcal I(A/J)\subseteq\pi(\mathcal I(A))$ for the quotient map
$\pi:A\to A/J$. For finite-dimensional $A$ we prove
$\mathcal I(A)=\operatorname{Sph}^+_A(E(A))$, where $E(A)$ is the set of
minimal projections of $A$ and $\operatorname{Sph}^+_A(E)$ is the set of
positive norm-one elements at distance one from every element of $E$;
we also show that the analogous statement with the minimal projections
of the bidual fails in general. When $\mathcal I(A)\neq\varnothing$ the
distance from $1$ to $\mathcal I(A)$ is $0$ or $1$, according to whether
$A$ has a nontrivial projection. Moreover, $\mathcal I(A)$ is norm
compact if and only if $A$ is finite-dimensional, closedness of
$\mathcal I(A)$ passes to hereditary subalgebras, $\mathcal I(A)$ is
norm closed for every von Neumann algebra, and the extreme points of
$\mathcal I(A)$ are exactly its projections. For a compact, locally path
connected space $X$ the set $\mathcal I(C(X))$ is closed, whereas
$\mathcal I(\ell_c(A))$ fails to be closed when $A$ is unital without
nontrivial projections.
\end{abstract}

\maketitle

\section{Introduction}

Let $A$ be a $C^*$-algebra, and let $A_+$ denote its positive cone and
$S(A)_+=\{a\in A_+:\|a\|=1\}$ the positive part of its unit sphere.
Consider
\[
\mathcal D(A)=\{(a,b)\in A_+^2:\ ab=0,\ \|a\|=\|b\|=1\},
\qquad
\phi_A:\mathcal D(A)\to S(A)_+,\quad \phi_A(a,b)=a+b,
\]
and put
\[
\mathcal I(A)=\operatorname{Im}\phi_A .
\]
Thus $\mathcal I(A)$ consists exactly of the positive norm-one elements
which can be written as a sum of two orthogonal positive norm-one
elements. The set $\mathcal I(A)$ is invariant under all
$*$-isomorphisms, in particular under unitary conjugation, and so it is
a natural isomorphism invariant of $A$. The purpose of this paper is to
investigate which properties of $A$ are reflected in the algebraic,
metric and topological structure of $\mathcal I(A)$.

\medskip
\noindent\textbf{Metric properties.} For a unital $C^*$-algebra $A$,
put $d(A)=\operatorname{dist}(1,\mathcal I(A))$. When
$\mathcal I(A)\ne\varnothing$ we have $d(A)\in\{0,1\}$, and $d(A)=0$ if
and only if $A$ contains a nontrivial projection
(Theorem~\ref{thm:main}(1)). This dichotomy was already observed in the
author's earlier note \cite{Taghavi2013}; a complete proof is included
here for the reader's convenience. A second metric description of
$\mathcal I(A)$ is obtained by comparing it with the sets
\[
\operatorname{Sph}^+_A(E)=\{x\in S(A)_+:\ \|x-e\|=1\text{ for every }e\in E\},
\qquad E\subseteq S(A)_+ .
\]
Such sphere conditions are the basis of a line of work on geometric
characterizations of projections among positive norm-one elements,
initiated by Nagy \cite{Nagy2013,Nagy2018} and continued by Peralta
\cite{Peralta2018} and by Peralta and Saavedra
\cite{PeraltaSaavedra2026}. Our contribution is of a different nature
and is tied to the orthogonal decomposition defining $\mathcal I(A)$: if
$A$ is finite-dimensional and $E(A)$ is the set of minimal projections
of $A$, then
\[
\mathcal I(A)=\operatorname{Sph}^+_A(E(A)),
\]
and both sets coincide with the set of $x\in S(A)_+$ for which the
eigenvalue $1$ has multiplicity at least two
(Proposition~\ref{prop:sph}). The finite-dimensional hypothesis cannot
simply be dropped, since $E(A)$ is empty for many infinite-dimensional
algebras; Remark~\ref{rem:bidual} below shows that replacing $E(A)$ by
the minimal projections of the bidual $A^{**}$ does not repair the
statement either.

\medskip
\noindent\textbf{Algebraic constructions.} In Section~2 we describe
$\mathcal I(A\oplus B)$ in terms of $\mathcal I(A)$, $\mathcal I(B)$ and
the positive unit balls of $A$ and $B$, and we show that
$\mathcal I(A/J)\subseteq\pi(\mathcal I(A))$ for every closed ideal $J$,
with an example showing that the reverse inclusion fails.

\medskip
\noindent\textbf{Topological properties.} Section~3 is devoted to the
norm topology of $\mathcal I(A)$. Theorem~\ref{thm:main} shows that
$\mathcal I(A)$ is norm compact if and only if $A$ is finite-dimensional;
that $\mathcal I(A)$ is finite and nonempty only for $A\cong\mathbb C^2$
or $A\cong M_2(\mathbb C)$ (and then $\mathcal I(A)=\{1\}$); that
closedness of $\mathcal I(A)$ passes to hereditary subalgebras; that
$\mathcal I(A)$ is norm closed whenever $A$ is a von Neumann algebra;
and that the extreme points of $\mathcal I(A)$ are precisely its
projections. For commutative algebras, Theorem~\ref{thm:CX} shows that
$\mathcal I(C(X))$ is closed whenever $X$ is a compact, locally path
connected space, and Theorem~\ref{thm:ell} gives closedness results for
$\mathcal I(\ell^\infty(A))$ together with an example, based on
$\ell_c(A)$, in which closedness fails. Several questions are
collected in the remarks.

\medskip
\noindent\textbf{Scope.} The graph of positive zero divisors and the
zero-divisor real rank, which were also introduced in \cite{Taghavi2013},
are not treated here; they are the subject of a separate paper.

\section{Positive  zero  divisors and $\mathcal I(A)$}

In this section, we associate with a unital $C^*$-algebra $A$,  three basic objects—the space $D(A)$ of normalized orthogonal positive pairs, the map $\phi_A$ and its image $\mathcal I(A)$. This construction will serve as the foundation for the algebraic, topological, and geometric investigations developed in the subsequent sections.

Let  $A$  be  a  $C^*$ algebra.  We  denote by  $A_{+}$ its  cone of  positive  elements  and  put  $$S(A)_+=\{a\in  A_+:\parallel a\parallel=1\}$$ We define $$\mathcal D(A)=\{(a,b)\in A_{+}^2\mid ab=0, \|a\|=\|b\|=1\}$$
We  consider  the  map  $$\phi_A:\mathcal D(A)\to  S(A)_{+},\qquad \phi_A(a,b)=a+b$$

For  $(a,b)\in \mathcal D(A)$ the  orthogonality relation $ab=0$ implies $$\|a+b\|=\max \{\|a\|,\|b\|\}=1$$ Hence  $\phi_A$ indeed takes its  value  in  $S(A)_+$.  We  denote its image  by    $$\mathcal I(A)=\operatorname{Im}\phi_A.$$
We  define \[
d(A)=\operatorname{dist}(1,\mathcal I(A)).
\] The following elementary properties will be useful. If $\alpha:A\to B$  is  a $C^*$ algebra isomorphism then  $$\alpha(\mathcal I(A))=\mathcal I(B).$$ In particular, every automorphism of A preserves $\mathcal I(A)$ and hence $\mathcal I(A)$ is invariant under  unitary conjugation. Note that $\mathcal D(A)$ and $\mathcal I(A)$ make sense for every $C^*$-algebra $A$; the unit is needed only for $d(A)$.

The positive elements of the unit ball of an algebra $A$ will be denoted by
\[
A_{1+}=\{a\in A_+:\|a\|\leq 1\}.
\]
We next record the behavior of $\mathcal I(A)$ under direct sums:
\begin{proposition}\label{prop:direct-sum}Let $A$ and $B$ be unital $C^*$-algebras. Then
\[
\begin{aligned}
\mathcal I(A\oplus B)
={}&
\bigl(\mathcal I(A)\times B_{1+}\bigr)
\cup
\bigl(A_{1+}\times\mathcal I(B)\bigr)\\
&\cup
\bigl(S(A)_+\times S(B)_+\bigr).
\end{aligned}
\]
\end{proposition}

\begin{proof}
Let $(x,y)\in\mathcal I(A\oplus B)$, say $(x,y)=(a_1,a_2)+(b_1,b_2)$ with
$(a_1,a_2),(b_1,b_2)\in(A\oplus B)_+$ of norm one and
$(a_1,a_2)(b_1,b_2)=0$. Then $a_1b_1=0$, $a_2b_2=0$, all four elements
are positive contractions, and
$\max\{\|a_1\|,\|a_2\|\}=\max\{\|b_1\|,\|b_2\|\}=1$. If
$\|a_1\|=\|b_1\|=1$, then $x=a_1+b_1\in\mathcal I(A)$ and
$y=a_2+b_2\in B_{1+}$. If $\|a_2\|=\|b_2\|=1$, then
$y\in\mathcal I(B)$ and $x\in A_{1+}$. In the two remaining cases
($\|a_1\|=\|b_2\|=1$ or $\|a_2\|=\|b_1\|=1$) we have
$\|x\|=\max\{\|a_1\|,\|b_1\|\}=1$ and $\|y\|=\max\{\|a_2\|,\|b_2\|\}=1$,
so $(x,y)\in S(A)_+\times S(B)_+$.

Conversely, if $x=a+b\in\mathcal I(A)$ and $y\in B_{1+}$, then
$(x,y)=(a,y)+(b,0)$ with orthogonal norm-one positive summands; the case
$x\in A_{1+}$, $y\in\mathcal I(B)$ is symmetric; and if
$x\in S(A)_+$, $y\in S(B)_+$, then $(x,y)=(x,0)+(0,y)$.
\end{proof}

We next record a quotient property of \(\mathcal I(A)\), which will also be useful in the subsequent discussion of closedness.

\begin{lemma}\label{lem:lift}
Let $A$ be a $C^*$-algebra, let $J$ be a closed two-sided ideal of $A$,
let $\pi:A\to A/J$ be the quotient map, and let $e,f\in(A/J)_+$ satisfy
$ef=0$. Let $x,y\in A_+$ with $\pi(x)=e$ and $\pi(y)=f$, and put
\[
z=(x-y)_+,\qquad w=(x-y)_-.
\]
Then $z,w\in A_+$, $zw=0$, $\pi(z)=e$, $\pi(w)=f$, $\|z\|\leq\|x\|$ and
$\|w\|\leq\|y\|$.
\end{lemma}

\begin{proof}
That $z,w\geq0$ and $zw=0$ is standard. By the compatibility of the
continuous functional calculus with $*$-homomorphisms,
$\pi(z)=(\pi(x-y))_+=(e-f)_+$ and $\pi(w)=(e-f)_-$. Since $e,f\geq0$
commute and $ef=0$, every character $\chi$ of the commutative
$C^*$-algebra $C^*(e,f)$ satisfies $\chi(e)\chi(f)=0$, so
$(\chi(e)-\chi(f))_+=\chi(e)$ and $(\chi(e)-\chi(f))_-=\chi(f)$. Hence
$(e-f)_+=e$ and $(e-f)_-=f$. Finally, $x-y\leq x$ gives
$\sigma(x-y)\subseteq(-\infty,\|x\|]$, hence $\|z\|\leq\|x\|$, and
$y-x\leq y$ gives $\|w\|=\|(y-x)_+\|\leq\|y\|$.
\end{proof}

\begin{proposition}[Lifting positive zero-divisor sums through quotients]
Let $A$ be a $C^*$-algebra, let $J$ be a closed two-sided ideal of $A$,
and let
\[
\pi:A\longrightarrow A/J
\]
be the canonical quotient map. Then
\[
\mathcal I(A/J)\subseteq \pi(\mathcal I(A)).
\]
\end{proposition}

\begin{proof}
Let $b\in\mathcal I(A/J)$. Then there exist
$e,f\in(A/J)_+$ such that
\[
b=e+f,\qquad ef=0,\qquad \|e\|=\|f\|=1.
\]
Choose positive contractions $x,y\in A$ such that
$\pi(x)=e$ and $\pi(y)=f$. The use of positive and negative parts to
obtain orthogonal positive elements is closely related to the lifting
methods in the theory of ideal perturbations; see \cite{AP,Lee2004}.
Let $z=(x-y)_+$ and $w=(x-y)_-$. By Lemma~\ref{lem:lift},
$z,w\in A_+$, $zw=0$, $\pi(z)=e$, $\pi(w)=f$ and
$\|z\|,\|w\|\leq1$. Consequently,
\[
1=\|e\|=\|\pi(z)\|\leq\|z\|\leq1
\quad\text{and}\quad
1=\|f\|=\|\pi(w)\|\leq\|w\|\leq1,
\]
so $\|z\|=\|w\|=1$. Thus $z+w\in\mathcal I(A)$ and
$\pi(z+w)=e+f=b$. Therefore $b\in\pi(\mathcal I(A))$.
\end{proof}

\begin{remark}
The reverse inclusion does not hold in general. For example, let
$A=\mathbb C^2$ and $J=\mathbb C\oplus0$, so that
$A/J\cong\mathbb C$ and $\mathcal I(\mathbb C)=\varnothing$, whereas
$1=(1,0)+(0,1)\in\mathcal I(A)$ and $\pi(1)\neq0$; hence
$\pi(\mathcal I(A))\neq\mathcal I(A/J)$. In general the quotient map may
decrease the norm of one of the two orthogonal positive summands
occurring in an element of $\mathcal I(A)$. The normalization condition
$\|a\|=\|b\|=1$ in the definition of $\mathcal I(A)$ is essential for
this distinction.
\end{remark}

The preceding results describe several elementary features of the set
$I(A)$. We next compare this set with a metric characterization of
positive norm-one elements in terms of minimal projections. The
underlying idea is that the algebraic decomposition defining $I(A)$
imposes a multiplicity condition on the maximal spectral value of a
positive norm-one element. In finite dimensions, this condition can be
detected directly through distances from minimal projections. This
connects the zero-divisor description of $I(A)$ with the metric
properties of positive operators studied in
\cite{Peralta2018,PeraltaSaavedra2026}.

For a subset $E\subseteq S(A)_+$, we use the notation

$$
\operatorname{Sph}_A^+(E)
=
\bigl\{
x\in S(A)_+:\|x-e\|=1\text{ for every }e\in E
\bigr\}.
$$

In the finite-dimensional setting, taking $E$ to be the set of minimal
projections leads to an exact characterization of $I(A)$.

\begin{proposition}\label{prop:sph}
Let $A$ be a finite-dimensional $C^*$-algebra, faithfully represented as
a block-diagonal subalgebra of $B(H)$, where

$$
A\cong\bigoplus_{i=1}^k M_{n_i}(\mathbb C),
\qquad
H=\bigoplus_{i=1}^k\mathbb C^{n_i}.
$$

Let $E(A)$ denote the set of minimal projections of $A$. For
$x\in S(A)_+$, let

$$
V_1(x)=\{\xi\in H:x\xi=\xi\}
$$

be the eigenspace corresponding to the eigenvalue $1$. Then

$$
x\in I(A)
\quad\Longleftrightarrow\quad
\dim V_1(x)\geq2
\quad\Longleftrightarrow\quad
x\in\operatorname{Sph}_A^+(E(A)).
$$

Consequently,

$$
I(A)=\operatorname{Sph}_A^+(E(A)).
$$

\end{proposition}

\begin{proof}
Suppose first that $x\in I(A)$. Then there exist
$a,b\in A_+$ such that

$$
x=a+b,\qquad ab=0,\qquad \|a\|=\|b\|=1.
$$

Since $A$ is finite-dimensional, the norms of $a$ and $b$ are attained.
Thus there exist unit vectors $\xi,\eta\in H$ such that

$$
a\xi=\xi,\qquad b\eta=\eta.
$$

Since $ab=0$, we have

$$
b\xi=ba\xi=0,\qquad a\eta=ab\eta=0.
$$

Therefore

$$
x\xi=(a+b)\xi=\xi,
\qquad
x\eta=(a+b)\eta=\eta.
$$

Moreover,

$$
\langle\xi,\eta\rangle
=
\langle a\xi,\eta\rangle
=
\langle\xi,a\eta\rangle
=0.
$$

Hence $\xi$ and $\eta$ are linearly independent elements of $V_1(x)$,
and consequently

$$
\dim V_1(x)\geq2.
$$

Conversely, suppose that

$$
\dim V_1(x)\geq2.
$$

Since $x$ is positive and finite-dimensional, its spectral decomposition
has the form

$$
x=\sum_{\lambda\in\sigma(x)}\lambda q_\lambda,
$$

where the $q_\lambda$ are mutually orthogonal projections in $A$.
Since $|x|=1$, we have $1\in\sigma(x)$, and

$$
V_1(x)=q_1H.
$$

Thus $\operatorname{rank}q_1\geq2$. We may therefore write

$$
q_1=q_1'+q_1'',
$$

where $q_1'$ and $q_1''$ are nonzero orthogonal projections in $A$.
Define

$$
a=q_1'+\sum_{\lambda\neq1}\lambda q_\lambda,
\qquad
b=q_1''.
$$

Then $a,b\in A_+$,

$$
ab=0,
\qquad
a+b=x.
$$

Both $a$ and $b$ have $1$ in their spectra, so

$$
\|a\|=\|b\|=1.
$$

Hence $x\in I(A)$.

It remains to establish the metric characterization. Let
$p_\xi\in E(A)$ be a minimal projection. In the given block-diagonal
representation, there is a unit vector $\xi\in H$ such that

$$
p_\xi\zeta=\langle\zeta,\xi\rangle\xi
\qquad(\zeta\in H).
$$

For $x\in S(A)_+$, we have $0\leq x\leq1$. We claim that
\begin{equation}
\|x-p_\xi\|=1
\quad\Longleftrightarrow\quad
x\xi=0
\ \text{or}
V_1(x)\cap\xi^\perp\neq{0}.
\label{eq:metric-characterization}
\end{equation}

Indeed, if $x\xi=0$, then

$$
(x-p_\xi)\xi=-\xi,
$$

and hence $\|x-p_\xi\|=1$. On the other hand, if
$\eta\in V_1(x)\cap\xi^\perp$ is a unit vector, then

$$
(x-p_\xi)\eta=\eta,
$$

so again $|x-p_\xi|=1$.

For the converse, suppose that $|x-p_\xi|=1$. Since $x-p_\xi$ is
self-adjoint and

$$
-1\leq x-p_\xi\leq1,
$$

the number $1$ or the number $-1$ is an eigenvalue of $x-p_\xi$.

Suppose first that $-1$ is an eigenvalue. Let $\zeta$ be a unit vector
such that

$$
(x-p_\xi)\zeta=-\zeta.
$$

Taking the inner product with $\zeta$, we obtain

$$
\langle x\zeta,\zeta\rangle
-
|\langle\zeta,\xi\rangle|^2
=-1.
$$

Since $x\geq0$ and
$|\langle\zeta,\xi\rangle|^2\leq1$, equality forces

$$
\langle x\zeta,\zeta\rangle=0,
\qquad
|\langle\zeta,\xi\rangle|=1.
$$

Positivity of $x$ gives $x\zeta=0$, while the second equality implies
that $\zeta$ is a scalar multiple of $\xi$. Thus

$$
x\xi=0.
$$

Now suppose that $1$ is an eigenvalue of $x-p_\xi$. Let $\eta$ be a unit
vector satisfying

$$
(x-p_\xi)\eta=\eta.
$$

Then

$$
x\eta=\eta+\langle\eta,\xi\rangle\xi.
$$

Taking the inner product with $\eta$ gives

$$
\langle x\eta,\eta\rangle
=
1+|\langle\eta,\xi\rangle|^2.
$$

Since $0\leq x\leq1$,

$$
\langle x\eta,\eta\rangle\leq1.
$$

Consequently,

$$
\langle\eta,\xi\rangle=0,
$$

and hence

$$
\eta\in V_1(x)\cap\xi^\perp.
$$

This proves \eqref{eq:metric-characterization}.

Suppose now that $\dim V_1(x)\geq2$. For every
$p_\xi\in E(A)$, the space $V_1(x)$ has dimension at least two, whereas
$\xi^\perp$ has codimension one. Therefore

$$
V_1(x)\cap\xi^\perp\neq\{0\}.
$$

By \eqref{eq:metric-characterization},

$$
\|x-p_\xi\|=1
\qquad\text{for every }p_\xi\in E(A).
$$

Thus

$$
x\in\operatorname{Sph}_A^+(E(A)).
$$

Conversely, suppose that
$x\in\operatorname{Sph}_A^+(E(A))$. If
$\dim V_1(x)=1$, choose a unit vector $\xi$ spanning $V_1(x)$. Then
$p_\xi\in E(A)$. Since $x\xi=\xi$, we have $x\xi\neq0$, while

$$
V_1(x)\cap\xi^\perp=\{0\}.
$$

By \eqref{eq:metric-characterization},

$$
\|x-p_\xi\|\neq1,
$$

contradicting
$x\in\operatorname{Sph}_A^+(E(A))$. Hence

$$
\dim V_1(x)\geq2.
$$

The two equivalences have now been proved, and therefore

$$
I(A)=\operatorname{Sph}_A^+(E(A)).
$$

\end{proof}

\begin{remark}\label{rem:bidual}
The equality in Proposition~\ref{prop:sph} is genuinely tied to the
finite-dimensional setting. In an infinite-dimensional $C^*$-algebra,
the set of minimal projections of $A$ may be empty. For example, this
occurs for many infinite-dimensional algebras of interest, including
UHF algebras, $\mathrm{II}_1$ factors, and the Calkin algebra. In such
cases,

$$
E(A)=\varnothing,
$$

and hence

$$
\operatorname{Sph}_A^+(E(A))=S(A)_+
$$

vacuously, whereas $I(A)$ need not equal the whole positive unit sphere.

One might therefore ask whether the minimal projections of the bidual
$A^{**}$ can provide a suitable replacement for those of $A$. The
corresponding equality
\begin{equation}
I(A)
=
\operatorname{Sph}_{A^{**}}^+(E(A^{**}))\cap A
\tag{2.5}
\end{equation}
does not hold in general. Indeed, consider

$$
A=C([0,1]).
$$

The bidual $A^{**}$ contains minimal projections, and for every minimal
projection $p\in E(A^{**})$ one has $p\neq1$ and hence

$$
\|1-p\|=1.
$$

Consequently,

$$
1\in
\operatorname{Sph}_{A^{**}}^+(E(A^{**}))\cap A.
$$

On the other hand,

$$
1\notin I(C([0,1])).
$$

Indeed, suppose that

$$
1=a+b,
\qquad
a,b\in C([0,1])_+,
\qquad
ab=0,
\qquad
\|a\|=\|b\|=1.
$$

Then, for every $t\in[0,1]$,

$$
a(t)+b(t)=1,
\qquad
a(t)b(t)=0.
$$

Thus

$$
a(t)\in\{0,1\}
\qquad(t\in[0,1]).
$$

Since $a$ is continuous and $[0,1]$ is connected, $a$ must be constant.
Hence either $a=0$ or $a=1$, contradicting

$$
\|a\|=\|b\|=1.
$$

Therefore

$$
1\notin I(C([0,1])),
$$

and so

$$
I(A)\neq
\operatorname{Sph}_{A^{**}}^+(E(A^{**}))\cap A
$$

in general.

Thus the finite-dimensional metric characterization cannot, in
general, be recovered merely by passing from $A$ to its bidual.
\end{remark}

\section{Topological properties of  $\mathcal I(A)$}

In this section we investigate several properties of $\mathcal I(A)$. In particular we address the problem of closedness and compactness of $\mathcal I(A)$. Recall that
$$
\mathcal I(A)
=
\left\{
a+b:\ a,b\in A_+,\ ab=0,\ \|a\|=\|b\|=1
\right\}.
$$
Since $\mathcal I(A)$ is in general not convex, we say that
$x\in\mathcal I(A)$ is an \emph{extreme point} of $\mathcal I(A)$, and write
$x\in\operatorname{Ext}(\mathcal I(A))$, if
$x=\frac{y+z}{2}$ with $y,z\in\mathcal I(A)$ implies $y=z\,(=x)$.
We write $\mathcal P(A)$ for the set of projections of $A$.

\begin{theorem}\label{thm:main}
Let $A$ be a unital $C^*$-algebra. Then the following assertions hold.
\begin{enumerate}
\item[(1)] If $\mathcal I(A)\neq\varnothing$, then

$$
d(A)\in\{0,1\}.
$$

More precisely,

$$
d(A)=0
\quad\Longleftrightarrow\quad
A\text{ contains a nontrivial projection}.
$$

In particular, if $A$ is projectionless, then $d(A)=1$.

\item[(2)] If $\mathcal I(A)$ is norm closed in $A$ and $B$ is a hereditary $C^*$-subalgebra of $A$, then $\mathcal I(B)$ is norm closed in $B$.

\item[(3)] If $\mathcal I(A)$ is finite and nonempty, then

$$
A\cong\mathbb C^2
\qquad\text{or}\qquad
A\cong M_2(\mathbb C).
$$

In both cases,

$$
\mathcal I(A)=\{1\}.
$$

\item[(4)] $\mathcal I(A)$ is norm compact if and only if $A$ is finite-dimensional.

\item[(5)] If $A$ is a von Neumann algebra, then $\mathcal I(A)$ is norm closed.

\item[(6)]
\[
\operatorname{Ext}\bigl(\mathcal I(A)\bigr)
=
\mathcal I(A)\cap\mathcal P(A).
\]
Equivalently, an element \(x\in\mathcal I(A)\) is an extreme point of
\(\mathcal I(A)\) if and only if \(x\) is a projection.
\end{enumerate}
\end{theorem}

\begin{proof}
\textbf{(1).}
Let $x\in\mathcal I(A)$. Then

$$
x=a+b,
\qquad
a,b\in A_+,
\qquad
ab=0,
\qquad
\|a\|=\|b\|=1.
$$

Since $a,b\geq0$ and $ab=0$, we have

$$
0\leq x\leq1
$$

and

$$
\|x\|=1.
$$

Suppose that $x$ is invertible. Since

$$
ax=a(a+b)=a^2
$$

and

$$
xa=(a+b)a=a^2,
$$

the elements $a$ and $x$ commute. Put

$$
p=ax^{-1}.
$$

Then $p=p^*$ and

$$
p^2=a^2x^{-2}
=axx^{-2}
=ax^{-1}
=p.
$$

Thus $p$ is a projection. Moreover, $p\neq0$ since $a\neq0$, and $p\neq1$ since $b\neq0$. Hence $A$ contains a nontrivial projection.

Consequently, if $A$ is projectionless, no element of $\mathcal I(A)$ is invertible. Thus, for every $x\in\mathcal I(A)$,

$$
0\in\sigma(x).
$$

Since $0\leq x\leq1$ and $\|x\|=1$, we also have

$$
1\in\sigma(x).
$$

Therefore

$$
\|1-x\|
=
\max_{\lambda\in\sigma(x)}|1-\lambda|
=1.
$$

Hence

$$
d(A)=1.
$$

Conversely, suppose that $A$ contains a nontrivial projection $p$. Then

$$
p,\ 1-p\geq0,\qquad
p(1-p)=0,
\qquad
\|p\|=\|1-p\|=1.
$$

Thus

$$
1=p+(1-p)\in\mathcal I(A),
$$

and consequently

$$
d(A)=0.
$$

This proves (1).

\medskip

\textbf{(2).}
Let $(x_n)$ be a sequence in $\mathcal I(B)$ converging in norm to some $x\in B$. Since

$$
\mathcal I(B)\subseteq\mathcal I(A)
$$

and $\mathcal I(A)$ is norm closed, we have

$$
x\in\mathcal I(A).
$$

Hence there exist $a,b\in A_+$ such that

$$
x=a+b,\qquad
ab=0,\qquad
\|a\|=\|b\|=1.
$$

Since

$$
0\leq a\leq x
\qquad\text{and}\qquad
0\leq b\leq x,
$$

and since $B$ is hereditary and $x\in B$, it follows that

$$
a,b\in B.
$$

Therefore

$$
x=a+b\in\mathcal I(B).
$$

Thus $\mathcal I(B)$ is norm closed in $B$.

\medskip

\textbf{(3).}
Assume that $\mathcal I(A)$ is finite and nonempty.

We first prove that $A$ is finite-dimensional. Suppose, to the contrary, that $A$ is infinite-dimensional. Then there exists a self-adjoint element

$$
h\in A_{\mathrm{sa}}
$$

with infinite spectrum. Choose three distinct points

$$
\lambda_1,\lambda_2,\lambda_3\in\sigma(h)
$$

and pairwise disjoint open neighborhoods

$$
U_1,U_2,U_3
$$

of these points. By continuous functional calculus, choose functions

$$
f_j\in C(\sigma(h))_+,
\qquad
\|f_j\|=1,
\qquad
\operatorname{supp}(f_j)\subseteq U_j
$$

for $j=1,2,3$. Put

$$
p=f_1(h),\qquad q=f_2(h),\qquad r=f_3(h).
$$

Then $p,q,r$ are nonzero positive elements of norm one and

$$
pq=pr=qr=0.
$$

For $t\in[0,1]$, define

$$
a_t=p+tq,
\qquad
b=r.
$$

Then

$$
a_t,b\geq0,\qquad
a_tb=0,\qquad
\|a_t\|=\|b\|=1.
$$

Hence

$$
p+tq+r\in\mathcal I(A)
$$

for every $t\in[0,1]$. These elements are distinct for distinct $t$, contradicting the finiteness of $\mathcal I(A)$.

Thus $A$ is finite-dimensional. Hence

$$
A\cong\bigoplus_{j=1}^rM_{n_j}(\mathbb C).
$$

Suppose first that $r\geq3$. Choose rank-one projections

$$
p_1,p_2,p_3
$$

in three distinct summands. For $t\in[0,1]$, put

$$
a_t=p_1+t p_3,
\qquad
b=p_2.
$$

Then

$$
a_t,b\geq0,\qquad
a_tb=0,\qquad
\|a_t\|=\|b\|=1,
$$

so

$$
p_1+t p_3+p_2\in\mathcal I(A).
$$

Again $\mathcal I(A)$ is infinite. Therefore

$$
r\leq2.
$$

Suppose that $r=2$. If one of the summands has dimension at least $2$, say

$$
n_1\geq2,
$$

choose orthogonal rank-one projections $p,q$ in the first summand and let $e$ be the identity of the second summand. Then

$$
a_t=p+t e,
\qquad
b=q
$$

satisfy

$$
a_t,b\geq0,\qquad
a_tb=0,\qquad
\|a_t\|=\|b\|=1.
$$

Thus

$$
p+t e+q\in\mathcal I(A)
$$

for every $t\in[0,1]$, a contradiction. Hence

$$
n_1=n_2=1,
$$

and therefore

$$
A\cong\mathbb C^2.
$$

Finally, suppose that $r=1$. Then

$$
A\cong M_n(\mathbb C).
$$

If $n\geq3$, choose three mutually orthogonal rank-one projections

$$
p,q,s.
$$

Then

$$
a_t=p+tq,
\qquad
b=s
$$

again gives infinitely many distinct elements

$$
p+tq+s\in\mathcal I(A).
$$

Hence $n\leq2$. Since $\mathcal I(A)\neq\varnothing$, we cannot have $A\cong\mathbb C$, and therefore

$$
A\cong M_2(\mathbb C).
$$

It remains to determine $\mathcal I(A)$ in these two cases.

If $A=\mathbb C^2$, the only two nonzero positive elements of norm one which are orthogonal are the two coordinate projections. Hence

$$
\mathcal I(\mathbb C^2)=\{1\}.
$$

If $A=M_2(\mathbb C)$ and

$$
x=a+b\in\mathcal I(A),
$$

then $a$ and $b$ are nonzero positive orthogonal elements. Their support projections are nonzero and orthogonal. Since the ambient space is two-dimensional, both supports are one-dimensional. Therefore $a$ and $b$ are positive operators supported on one-dimensional subspaces and having norm one; hence they are rank-one projections. Their supports are complementary, so

$$
a+b=1.
$$

Thus

$$
\mathcal I(M_2(\mathbb C))=\{1\}.
$$

This proves (3).

\medskip

\textbf{(4).}
Suppose first that $A$ is finite-dimensional. Consider

$$
\mathcal D
=
\left\{
(a,b)\in A_+\times A_+:
ab=0,\ \|a\|=\|b\|=1
\right\}.
$$

The set $\mathcal D$ is closed and bounded in the finite-dimensional space $A\times A$, hence compact. The map

$$
\mathcal D\longrightarrow A,
\qquad
(a,b)\longmapsto a+b
$$

is continuous, and its image is precisely $\mathcal I(A)$. Hence $\mathcal I(A)$ is compact.

Conversely, suppose that $A$ is infinite-dimensional. Choose

$$
h\in A_{\mathrm{sa}}
$$

with infinite spectrum. Select infinitely many distinct points of $\sigma(h)$ and pairwise disjoint open neighborhoods of these points. By continuous functional calculus, there exist positive elements

$$
x_1,x_2,\ldots\in A
$$

such that

$$
\|x_n\|=1
\qquad\text{and}\qquad
x_nx_m=0
\quad(n\neq m).
$$

For $n\geq2$, put

$$
y_n=x_1+x_n.
$$

Then

$$
y_n\in\mathcal I(A).
$$

For $n\neq m$,

$$
y_n-y_m=x_n-x_m.
$$

Since $x_n,x_m\geq0$ are orthogonal and have norm one,

$$
\|x_n-x_m\|=1.
$$

Therefore

$$
\|y_n-y_m\|=1
\qquad(n\neq m).
$$

Thus $(y_n)$ has no norm-convergent subsequence, and $\mathcal I(A)$ cannot be norm compact.

Hence $\mathcal I(A)$ is norm compact if and only if $A$ is finite-dimensional.

\medskip

\textbf{(5).}
Assume that $A$ is a von Neumann algebra.

We first establish the following claim.

\medskip

\emph{Claim.}
Let $x\in A_+$ with $\|x\|=1$. Then
\[
x\notin\mathcal I(A)
\]
if and only if $1$ is an isolated point of $\sigma(x)$ and the spectral
projection
\[
E_x(\{1\})
\]
is a minimal projection.

\medskip

Suppose first that $1$ is not isolated in $\sigma(x)$. Choose pairwise
disjoint open intervals
\[
J_1,J_2,\ldots\subset(0,1)
\]
such that
\[
J_n\cap\sigma(x)\neq\varnothing
\]
for every $n$, and such that the intervals accumulate at $1$. Put
\[
e_n=E_x(J_n).
\]
Since the support of the spectral measure $E_x$ is precisely
$\sigma(x)$, every open set which meets $\sigma(x)$ has nonzero spectral
projection. Hence
\[
e_n\neq0
\qquad(n\geq1).
\]
Moreover, since the intervals $J_n$ are pairwise disjoint, the
projections $e_n$ are mutually orthogonal.

Define
\[
P=\sum_{n\ {\rm even}}e_n.
\]
The sum is understood in the strong operator topology. Indeed, the
finite partial sums
\[
P_N=\sum_{\substack{n\leq N\\ n\ {\rm even}}}e_n
\]
are increasing projections satisfying $0\leq P_N\leq1$. Since $A$ is a
von Neumann algebra, $P_N$ converges strongly to a projection
$P\in A$. Similarly,
\[
Q=\sum_{n\ {\rm odd}}e_n
\]
converges strongly to a projection $Q\in A$, and $P$ and $Q$ are
orthogonal.

We claim that
\[
\|xP\|=1
\qquad\text{and}\qquad
\|xQ\|=1.
\]
Indeed, given $\varepsilon>0$, since the intervals $J_n$ accumulate at
$1$, there is an even $n$ such that
\[
J_n\subset(1-\varepsilon,1).
\]
Since $e_n\neq0$ and $e_n=E_x(J_n)$, the restriction of $x$ to
$e_nH$ has norm at least $1-\varepsilon$. As $e_n\leq P$, it follows that
\[
\|xP\|\geq 1-\varepsilon.
\]
Since $\|xP\|\leq\|x\|=1$, we obtain $\|xP\|=1$. The same argument,
using an odd $n$, gives $\|xQ\|=1$.

Since $Q\leq1-P$, we also have
\[
\|x(1-P)\|\geq\|xQ\|=1.
\]
On the other hand,
\[
\|x(1-P)\|\leq\|x\|=1.
\]
Therefore
\[
\|x(1-P)\|=1.
\]

Now $P$ is a spectral projection of $x$, so $P$ commutes with $x$.
Consequently,
\[
x=xP+x(1-P),
\]
where
\[
xP\geq0,\qquad x(1-P)\geq0,
\]
and
\[
xP\,x(1-P)=0.
\]
Both summands have norm $1$. Hence
\[
x\in\mathcal I(A).
\]
Thus
\[
x\notin\mathcal I(A)
\quad\Longrightarrow\quad
1\text{ is isolated in }\sigma(x).
\]

Now suppose that $1$ is isolated in $\sigma(x)$ and put
\[
p=E_x(\{1\}).
\]
If $p$ is not minimal, choose a nonzero proper projection
\[
0<q<p.
\]
Since $q\leq p=E_x(\{1\})$, we have
\[
xq=q.
\]
Moreover, $q$ commutes with $x$. Define
\[
a=q,
\qquad
b=x(1-q).
\]
Then
\[
a,b\geq0,
\qquad
ab=0,
\qquad
\|a\|=1.
\]
Since $p-q\neq0$ and $x$ acts as the identity on $pH$, we have
\[
\|b\|=1.
\]
Finally,
\[
a+b=q+x(1-q)=x.
\]
Therefore
\[
x\in\mathcal I(A).
\]
It follows that
\[
x\notin\mathcal I(A)
\quad\Longrightarrow\quad
E_x(\{1\})\text{ is minimal}.
\]

Conversely, suppose that
\[
x=a+b\in\mathcal I(A),
\]
where
\[
a,b\geq0,\qquad
ab=0,\qquad
\|a\|=\|b\|=1.
\]
Then $a$ and $b$ commute and
\[
x=a+b.
\]
Suppose that $1$ is isolated in $\sigma(x)$. Since
\[
\sigma(a),\sigma(b)\subset[0,1]
\]
and
\[
\|a\|=\|b\|=1,
\]
we have
\[
1\in\sigma(a)\cap\sigma(b).
\]
Let $C=C^*(a,b)$, a commutative $C^*$-algebra. If $\chi$ is a character
of $C$, then $\chi(a)\chi(b)=\chi(ab)=0$, so at most one of
$\chi(a),\chi(b)$ is nonzero and $\chi(x)$ equals that one. Hence
$\sigma(a)\setminus\{0\}\subseteq\sigma(x)$ and
$\sigma(b)\setminus\{0\}\subseteq\sigma(x)$. Since $1$ is isolated in
$\sigma(x)$, it is isolated in $\sigma(a)$ and in $\sigma(b)$. Choose
a continuous function $h$ on $[0,1]$ with $h(0)=0$, $h(1)=1$ and
$h=0$ on $\sigma(x)\setminus\{1\}$. Then $h(a)=E_a(\{1\})$,
$h(b)=E_b(\{1\})$ and $h(x)=E_x(\{1\})$; moreover, because $ab=0$ we
have $x^k=a^k+b^k$ for all $k\geq1$, hence
$h(x)=h(a)+h(b)$. Therefore
\[
E_x(\{1\})=E_a(\{1\})+E_b(\{1\}).
\]
Since $1\in\sigma(a)\cap\sigma(b)$ is an isolated point of each spectrum,
\[
E_a(\{1\})\neq0,
\qquad
E_b(\{1\})\neq0.
\]
Both projections lie in $C$ and are of the form $f(a)$, $g(b)$ with
$f(0)=g(0)=0$; since $ab=0$ their product vanishes, so they are
orthogonal.
Thus $E_x(\{1\})$ is the sum of two nonzero orthogonal projections and
therefore is not minimal.

Hence, if $1$ is isolated in $\sigma(x)$ and
$E_x(\{1\})$ is minimal, then
\[
x\notin\mathcal I(A).
\]
This proves the claim.

\medskip

We now prove that $\mathcal I(A)$ is norm closed. Let
\[
x_n\in\mathcal I(A)
\]
and suppose that
\[
x_n\longrightarrow x
\]
in norm. Since $\mathcal I(A)\subseteq A_+$ and the positive cone is
norm closed, we have
\[
x\in A_+.
\]
Moreover,
\[
\|x\|=\lim_n\|x_n\|=1.
\]

Suppose, toward a contradiction, that
\[
x\notin\mathcal I(A).
\]
By the claim, $1$ is isolated in $\sigma(x)$ and
\[
p=E_x(\{1\})
\]
is a minimal projection.

Choose $r$ with $0<r<1$ such that
\[
\sigma(x)\cap[1-r,1]=\{1\}.
\]
Let
\[
\Gamma=\{z\in\mathbb C:|z-1|=r\}.
\]
Its only real points are $1\pm r$, and $1-r\notin\sigma(x)\subseteq[0,1]$, so $\Gamma$ does not
meet $\sigma(x)$. Then $\Gamma\subset\rho(x)$ and
\[
p=
\frac{1}{2\pi i}
\int_\Gamma (z-x)^{-1}\,dz.
\]
by the Riesz projection formula (see, e.g.,
\cite[Theorem~III.6.17, p.~178]{KATO}).

Since $x_n\to x$ in norm, for all sufficiently large $n$ we have
$\Gamma\subset\rho(x_n)$, and the corresponding Riesz projections
\[
p_n=
\frac{1}{2\pi i}
\int_\Gamma (z-x_n)^{-1}\,dz
\]
satisfy
\[
\|p_n-p\|\longrightarrow0.
\]
In particular, for sufficiently large $n$,
\[
\|p_n-p\|<1.
\]
It follows that $p_n$ and $p$ are unitarily equivalent by a unitary of $A$ (a standard fact for two projections of a von Neumann algebra at distance less than one). Since $p$ is
minimal, $p_n$ is also minimal.

On the other hand, write
\[
x_n=a_n+b_n,
\]
where
\[
a_n,b_n\geq0,\qquad
a_nb_n=0,\qquad
\|a_n\|=\|b_n\|=1.
\]
Choose $n$ sufficiently large so that the above construction of $p_n$
applies. Since
\[
\|a_n\|=\|b_n\|=1,
\]
we have
\[
E_{a_n}((1-r,1])\neq0
\qquad\text{and}\qquad
E_{b_n}((1-r,1])\neq0.
\]
Because $a_nb_n=0$, these two spectral projections are orthogonal.

Since $\sigma(x_n)\subseteq[0,1]$ and $0<r<1$, the Riesz projection $p_n$
is the spectral projection $E_{x_n}((1-r,1])$. Because $a_nb_n=0$, we have
$x_n^k=a_n^k+b_n^k$ for all $k\geq1$; hence
$h(x_n)=h(a_n)+h(b_n)$ for every continuous $h$ on $[0,1]$ with
$h(0)=0$, and then, by the usual monotone class argument in the von
Neumann algebra $A$, for every bounded Borel function $h$ with $h(0)=0$.
Applying this to $h=\chi_{(1-r,1]}$ gives
\[
p_n
=
E_{a_n}((1-r,1])
+
E_{b_n}((1-r,1]).
\]
Thus $p_n$ is the sum of two nonzero orthogonal projections. In
particular, $p_n$ is not minimal, contradicting the minimality of
$p_n$ obtained above.

Therefore
\[
x\in\mathcal I(A).
\]
Hence $\mathcal I(A)$ is norm closed.

\noindent
\textbf{Proof of (6).}
We prove the two inclusions separately.

\medskip

\noindent
\textbf{First, suppose that \(x\in\operatorname{Ext}(\mathcal I(A))\).}
By the definition of \(\mathcal I(A)\), there exist \(a,b\in A_+\) such that
\[
x=a+b,\qquad ab=0,\qquad \|a\|=\|b\|=1.
\]
Since \(a\) and \(b\) are positive and orthogonal, we have
\[
x^2=(a+b)^2=a^2+b^2.
\]
Consequently,
\[
x-x^2=(a-a^2)+(b-b^2).
\]
Suppose, towards a contradiction, that \(x\) is not a projection. Then
\(x-x^2\neq0\), and hence at least one of \(a-a^2\) and \(b-b^2\)
is nonzero. Without loss of generality, assume that
\[
a-a^2\neq0.
\]

Choose \(t\) with \(0<t<1\), and put
\[
a_\pm=a\pm t(a-a^2).
\]
We claim that \(a_\pm\in A_+\) and \(\|a_\pm\|=1\).

Indeed, since \(a\geq0\) and \(\|a\|=1\),
\[
\sigma(a)\subseteq[0,1]
\qquad\text{and}\qquad
1\in\sigma(a).
\]
By continuous functional calculus,
\[
a_\pm=f_\pm(a),
\qquad
f_\pm(s)=s\pm ts(1-s).
\]
For \(0\leq s\leq1\),
\[
f_+(s)=s+ts(1-s)\geq0
\]
and
\[
1-f_+(s)=(1-s)(1-ts)\geq0.
\]
Thus
\[
0\leq f_+(s)\leq1.
\]
Similarly,
\[
f_-(s)=s-ts(1-s)=s(1-t+ts)\geq0
\]
and
\[
1-f_-(s)=(1-s)(1+ts)\geq0.
\]
Thus
\[
0\leq f_-(s)\leq1.
\]
Therefore
\[
0\leq a_\pm\leq1.
\]
Moreover,
\[
f_\pm(1)=1,
\]
and since \(1\in\sigma(a)\), the spectral mapping theorem gives
\[
1\in\sigma(a_\pm).
\]
Hence
\[
\|a_\pm\|=1.
\]

Since \(a_\pm\in C^*(a)\) and \(ab=0\), we have
\[
a_\pm b=ba_\pm=0.
\]
Therefore
\[
x_\pm:=a_\pm+b\in\mathcal I(A).
\]
Furthermore,
\[
\frac{x_++x_-}{2}
=
\frac{a_++a_-}{2}+b
=a+b=x.
\]
On the other hand,
\[
x_+-x_-=2t(a-a^2)\neq0,
\]
so \(x_+\neq x_-\). This contradicts the assumption that \(x\) is an
extreme point of \(\mathcal I(A)\). Thus \(x\) must be a projection.
Consequently,
\[
\operatorname{Ext}(\mathcal I(A))
\subseteq
\mathcal I(A)\cap\mathcal P(A).
\]

\medskip

\noindent
\textbf{Conversely, suppose that \(p\in\mathcal I(A)\cap\mathcal P(A)\).}
We prove that \(p\) is an extreme point of \(\mathcal I(A)\).

Suppose that
\[
p=\frac{x+y}{2},
\qquad x,y\in\mathcal I(A).
\]
Every element of \(\mathcal I(A)\) is positive and has norm \(1\).
Hence
\[
x,y\geq0,
\qquad
\|x\|=\|y\|=1.
\]
Since
\[
x+y=2p,
\]
we have
\[
0\leq x\leq2p,
\qquad
0\leq y\leq2p.
\]

We first show that \(x,y\in pAp\). Since \(0\leq x\leq2p\),
multiplication on both sides by \(1-p\) gives
\[
0\leq(1-p)x(1-p)
\leq2(1-p)p(1-p)=0.
\]
Hence
\[
(1-p)x(1-p)=0.
\]
Since \(x\geq0\), this implies
\[
x^{1/2}(1-p)=0,
\]
and therefore
\[
x(1-p)=(1-p)x=0.
\]
Thus
\[
x=pxp\in pAp.
\]
Similarly,
\[
y=pyp\in pAp.
\]

Now \(pAp\) is a unital \(C^*\)-algebra whose identity is \(p\).
Since \(x\geq0\), \(x\in pAp\), and
\[
\|x\|=1,
\]
the spectrum of \(x\), computed in \(pAp\), is contained in
\([0,1]\). Therefore
\[
0\leq x\leq p.
\]
Likewise,
\[
0\leq y\leq p.
\]

But
\[
x+y=2p.
\]
Since \(x\leq p\) and \(y\leq p\), we have
\[
2p-x-y\geq0.
\]
The left-hand side is zero, so equality must hold in both inequalities:
\[
x=p,\qquad y=p.
\]
Thus every representation
\[
p=\frac{x+y}{2},
\qquad x,y\in\mathcal I(A),
\]
is trivial. Hence \(p\) is an extreme point of \(\mathcal I(A)\).

Therefore,
\[
\mathcal I(A)\cap\mathcal P(A)
\subseteq
\operatorname{Ext}(\mathcal I(A)).
\]

Combining the two inclusions, we obtain
\[
\boxed{
\operatorname{Ext}(\mathcal I(A))
=
\mathcal I(A)\cap\mathcal P(A).
}
\]
This completes the proof of (6).
\end{proof}

\begin{remark}
The norm closedness in part {\rm (5)} does not, in general, imply weak-operator closedness.

Indeed, let $A$ be an infinite-dimensional von Neumann algebra. Then $A$ contains a sequence

$$
e_1,e_2,\ldots
$$

of nonzero pairwise orthogonal projections. In a faithful representation $A\subseteq B(H)$, these projections converge strongly, and hence weakly, to zero. Indeed, for every $\xi\in H$,

$$
\sum_{n=1}^{\infty}\|e_n\xi\|^2
\leq\|\xi\|^2,
$$

so

$$
e_n\xrightarrow{\mathrm{SOT}}0.
$$

Now put

$$
x_n=e_{2n-1}+e_{2n}.
$$

Then

$$
x_n\in\mathcal I(A)
$$

for every $n$, while

$$
x_n\xrightarrow{\mathrm{SOT}}0
$$

and hence

$$
x_n\xrightarrow{\mathrm{WOT}}0.
$$

Since

$$
0\notin\mathcal I(A),
$$

we conclude that $\mathcal I(A)$ is not weak-operator closed.

Thus, for an infinite-dimensional von Neumann algebra, $\mathcal I(A)$ is norm closed but need not be WOT closed.
\end{remark}

\begin{remark}
Part~(6) reveals a particularly significant geometric feature of the set
\(\mathcal I(A)\).  The condition of being a projection, which is an
algebraic condition, is recovered here purely from the convex geometry of
\(\mathcal I(A)\): the projections are exactly the points of
\(\mathcal I(A)\) which cannot be expressed as non-trivial midpoints of
two elements of \(\mathcal I(A)\).  Thus

$$
\operatorname{Ext}(\mathcal I(A))
=
\mathcal I(A)\cap\mathcal P(A)
$$

may be viewed as a geometric manifestation of idempotency.

This observation is suggestive in connection with the
Kadison--Kaplansky idempotent conjecture, which predicts that the reduced
group $C^*$-algebra $C^*_{\mathrm{red}}G$ of a torsion-free discrete group
$G$ has no projections other than $0$ and $1$. By Theorem~\ref{thm:main}(1),
the absence of nontrivial projections is detected by the invariant $d(A)$
(when $\mathcal I(A)\neq\varnothing$), and, by part~(6), projections are
exactly the extreme points of $\mathcal I(A)$. We do not pursue this
connection here.

For
\(x=a+b\in\mathcal I(A)\), with \(a,b\geq0\), \(ab=0\), and
\(\|a\|=\|b\|=1\), one has

$$
x^2=a^2+b^2
$$

and hence

$$
x-x^2=(a-a^2)+(b-b^2).
$$

Consequently, the failure of \(x\) to be idempotent is reflected in a
non-trivial ``idempotent defect'' inside the very decomposition defining
\(\mathcal I(A)\).  Part~(6) shows that this algebraic defect has a
precise convex-geometric counterpart: the non-idempotent elements of
\(\mathcal I(A)\) admit non-trivial convex decompositions, whereas the
idempotent elements are exactly the extreme ones.

In this sense, \(\mathcal I(A)\) provides a natural geometric arena in
which the idempotent phenomenon can be investigated.  The result also
recalls, in a particularly suggestive way, the classical characterization
of projections as the extreme points of the positive part of the unit
ball of a \(C^*\)-algebra.  Here the same geometric principle survives
after imposing the additional orthogonality and normalization conditions
defining \(\mathcal I(A)\).

This perspective suggests that the geometry of \(\mathcal I(A)\) may
contain information extending beyond the mere identification of its
extreme points.  In particular, it becomes natural to ask how symmetries
of \(A\), and especially continuous families of automorphisms, are
reflected in the topology and geometry of \(\mathcal I(A)\).  This leads
to a further direction of investigation concerning the action of
one-parameter groups of outer automorphisms on \(\mathcal I(A)\).
\end{remark}

The  next  theorem establishes closedness of  $\mathcal I(A)$  for  certain  commutative  $C^*$ algebra $A$.

\begin{theorem}\label{thm:CX}
Let $X$ be  a     locally  path  connected  compact Hausdorff space. Then $\mathcal I(C(X))$ is  a closed  set.
\end{theorem}
\begin{proof}
First we prove that a positive function $f\in C(X)$ of norm one belongs to $\mathcal I(C(X))$ if and only if $f^{-1}(\{1\})$ is not contained in any connected component of $X\setminus f^{-1}(\{0\})$. Let $f\in\mathcal I(C(X))$. Assume to the contrary that $f^{-1}(\{1\})$ is contained in a connected component $U$ of $X\setminus f^{-1}(\{0\})$. Write $f=g_1+g_2$, where $g_1,g_2$ are nonnegative continuous functions with $\|g_1\|=\|g_2\|=1$ and $g_1g_2=0$. On $U$ we have
\[
1=\frac{g_1(x)}{f(x)}+\frac{g_2(x)}{f(x)}.
\]
Since $g_1g_2=0$, each of the continuous functions $g_1/f$ and $g_2/f$ takes only the values $0$ and $1$. Hence they are constant on the connected set $U$, and exactly one of them is identically $1$ on $U$, while the other is identically $0$. Since $\|g_1\|=\|g_2\|=1$ and $X$ is compact, each $g_i$ assumes the value $1$. The two sets $g_1^{-1}(\{1\})$ and $g_2^{-1}(\{1\})$ are therefore nonempty, and they are contained in different connected components of $X\setminus f^{-1}(\{0\})$. This contradicts the assumption that $f^{-1}(\{1\})=g_1^{-1}(\{1\})\sqcup g_2^{-1}(\{1\})$ is contained in $U$.  In particular, if $X$ is connected, then a non-vanishing continuous function does not belong to $\mathcal I(C(X))$.

Conversely, assume that $f^{-1}(\{1\})$ is not contained in a connected component of $X\setminus f^{-1}(\{0\})$. Choose $a,b\in X$ with $f(a)=f(b)=1$ which belong to distinct connected components of $X\setminus f^{-1}(\{0\})$. Since $X$ is locally connected, the open set $X\setminus f^{-1}(\{0\})$ is the disjoint union of its open connected components. Let $W_1$ be the connected component containing $a$, and let $W_2$ be the union of all the remaining components. Then $W_1$ and $W_2$ are disjoint open sets whose union is $X\setminus f^{-1}(\{0\})$. Define
\[
g_1(x)=\begin{cases}f(x),&x\in\overline{W_1},\\0,&x\notin\overline{W_1},\end{cases}
\qquad
 g_2(x)=\begin{cases}f(x),&x\in\overline{W_2},\\0,&x\notin\overline{W_2}.\end{cases}
\]
Since $W_2$ is open and disjoint from $W_1$, we have $\overline{W_1}\subseteq W_1\cup f^{-1}(\{0\})$, and similarly $\overline{W_2}\subseteq W_2\cup f^{-1}(\{0\})$; hence $\overline{W_1}\cap\overline{W_2}\subseteq f^{-1}(\{0\})$. The function $g_1$ equals $f$ on the closed set $\overline{W_1}$ and $0$ on the closed set $X\setminus W_1$, and these two definitions agree on the intersection $\overline{W_1}\setminus W_1\subseteq f^{-1}(\{0\})$; so $g_1$ is continuous by the pasting lemma, and likewise $g_2$. Moreover, $g_1g_2=0$, $f=g_1+g_2$, and $\|g_1\|=\|g_2\|=1$. Thus $f\in\mathcal I(C(X))$.

It remains to prove that $\mathcal I(C(X))$ is closed under uniform convergence. Let $(f_n)$ be a sequence in $\mathcal I(C(X))$ converging uniformly to $f$. Then $f\geq0$ and $\|f\|=1$. By the characterization just proved, for every $n$ there are points $a_n,b_n\in X$ with
\[
f_n(a_n)=f_n(b_n)=1,
\]
lying in distinct connected components of $X\setminus f_n^{-1}(\{0\})$. Put $\varepsilon_n=\|f_n-f\|_\infty\to0$.

Suppose, toward a contradiction, that $f\notin\mathcal I(C(X))$. Then $M=f^{-1}(\{1\})\neq\varnothing$ is contained in a single connected component $C$ of the open set $U=X\setminus f^{-1}(\{0\})$. Since $X$ is locally path connected, $C$ is open and path connected. The set $X\setminus C$ is closed, hence compact, and $f<1$ on it; put $H=\max_{X\setminus C}f<1$ (and $H=0$ if $C=X$), and choose $\lambda$ with $H<\lambda<1$. Then
\[
K=\{x\in X:f(x)\geq\lambda\}
\]
is compact and $K\subseteq C$. Fix $x_0\in M$. For each $x\in K$ choose a path $\gamma_x:[0,1]\to C$ joining $x_0$ to $x$. Since $X$ is compact Hausdorff, hence regular, and locally path connected, we can choose a path connected open neighbourhood $V_x$ of $x$ with $\overline{V_x}\subseteq C$. The family $(V_x)_{x\in K}$ covers the compact set $K$, so finitely many, say $V_{x_1},\ldots,V_{x_m}$, cover $K$. Put
\[
L=\bigcup_{j=1}^m\gamma_{x_j}([0,1]),
\qquad
E=L\cup\bigcup_{j=1}^mV_{x_j}.
\]
Then $L$ is connected (all the paths start at $x_0$), each $V_{x_j}$ is connected and contains the point $x_j\in L$, so $E$ is connected, and hence so is its closure $\overline E\subseteq L\cup\bigcup_j\overline{V_{x_j}}\subseteq C$. The set $\overline E$ is a compact subset of $C\subseteq\{f>0\}$, so $f\geq\delta$ on $\overline E$ for some $\delta>0$. Consequently, for all $n$ with $\varepsilon_n<\delta$ we have $f_n>0$ on $\overline E$, i.e.\ $E\subseteq X\setminus f_n^{-1}(\{0\})$.

On the other hand, $f_n(a_n)=f_n(b_n)=1$ gives $f(a_n),f(b_n)\geq1-\varepsilon_n>\lambda$ for all large $n$, so $a_n,b_n\in K\subseteq\bigcup_jV_{x_j}\subseteq E$. Hence, for all sufficiently large $n$, both $a_n$ and $b_n$ belong to the connected set $E\subseteq X\setminus f_n^{-1}(\{0\})$, so they lie in the same connected component of $X\setminus f_n^{-1}(\{0\})$, contradicting the choice of $a_n,b_n$.

Thus $f\in\mathcal I(C(X))$, and $\mathcal I(C(X))$ is closed.
\end{proof}

\begin{remark}
Theorem~\ref{thm:CX} suggests a natural entry in the dictionary between topological properties of spaces and corresponding properties of commutative $C^*$-algebras. Namely, for a compact locally path connected space $X$, the set $\mathcal I(C(X))$ is closed. Thus, in the commutative setting, local path connectedness provides a sufficient topological condition for the closedness of $\mathcal I(C(X))$.

This condition is not necessary. For example, the spectrum of $\ell^\infty$ is extremely disconnected, and hence far from being locally path connected, while $\ell^\infty$ is a von Neumann algebra, so that $\mathcal I(\ell^\infty)$ is norm closed by Theorem~\ref{thm:main}(5). On the other hand, some hypothesis on $X$ is needed: for the compact space $X=\mathbb N\cup\{\infty\}$, which is not locally connected, $\mathcal I(C(X))$ is not closed (Theorem~\ref{thm:ell} below with $A=\mathbb C$). Thus Theorem~\ref{thm:CX} should be viewed as a sufficient topological criterion rather than as a characterization. Regarding closedness of $\mathcal I(\ell^\infty(A))$ we have the following more general result.
\end{remark}

\begin{theorem}\label{thm:ell}
Let $A$ be a $C^*$-algebra. Then the following assertions hold:
\begin{itemize}
\item If $\mathcal I(A)$ is norm closed, then $\mathcal I(\ell^{\infty}(A))$ is norm closed.
\item If $A$ is unital and has no nontrivial projection, then $\mathcal I(\ell_c(A))$ is not norm closed. Here $\ell_c(A)$ denotes the $C^*$-algebra of norm convergent sequences in $A$ with the supremum norm. In particular $\mathcal I(\ell_c(\mathbb C))=\mathcal I(C(\mathbb N\cup\{\infty\}))$ is not closed.
\end{itemize}
\end{theorem}

\begin{proof}
\emph{First assertion.} We first characterize $\mathcal I(\ell^\infty(A))$. Let $x=(x_n)\in\ell^\infty(A)$ with $x_n\geq0$ for all $n$ and $\|x\|=\sup_n\|x_n\|=1$. We claim that $x\in\mathcal I(\ell^\infty(A))$ if and only if one of the following holds:
\begin{itemize}
\item[(a)] $x_j\in\mathcal I(A)$ for some index $j$;
\item[(b)] for every $\delta>0$ there are two distinct indices $i\neq j$ with $\|x_i\|>1-\delta$ and $\|x_j\|>1-\delta$.
\end{itemize}

Suppose $x=\alpha+\beta$ with $\alpha,\beta\in\ell^\infty(A)_+$, $\alpha\beta=0$ and $\|\alpha\|=\|\beta\|=1$, and suppose that (b) fails. Then there is $\delta>0$ such that at most one index has $\|x_n\|>1-\delta$; since $\sup_n\|x_n\|=1$, exactly one index $j$ does. As $\|x_n\|=\max\{\|\alpha_n\|,\|\beta_n\|\}\leq1-\delta$ for $n\neq j$, the suprema $\sup_n\|\alpha_n\|=1$ and $\sup_n\|\beta_n\|=1$ can only be attained at $n=j$; thus $\|\alpha_j\|=\|\beta_j\|=1$, $\alpha_j\beta_j=0$, and $x_j=\alpha_j+\beta_j\in\mathcal I(A)$. So (a) holds.

Conversely, suppose (a) holds with $x_j=a+b$, $a,b\in A_+$, $ab=0$, $\|a\|=\|b\|=1$. Let $\alpha$ have $a$ in position $j$ and $0$ elsewhere, and let $\beta$ have $b$ in position $j$ and $x_n$ in position $n\neq j$. Then $\alpha,\beta\geq0$, $\alpha\beta=0$, $\|\alpha\|=1$, $\|\beta\|=1$ and $\alpha+\beta=x$. Suppose (b) holds. If there are $i\neq j$ with $\|x_i\|=\|x_j\|=1$, let $\alpha$ be $x_i$ in position $i$ and $0$ elsewhere, and $\beta=x-\alpha$. Otherwise at most one index (possibly none) has norm exactly $1$, and (b) yields pairwise distinct indices $n_1,n_2,\ldots$, all different from that index, with $\|x_{n_k}\|\to1$. Let $\alpha$ agree with $x$ at the indices $n_{2k}$ and vanish elsewhere, and $\beta=x-\alpha$. Then $\alpha\beta=0$, $\|\alpha\|=1$, and $\|\beta\|\geq\sup_k\|x_{n_{2k+1}}\|=1$, so $x\in\mathcal I(\ell^\infty(A))$. This proves the characterization.

Now assume $\mathcal I(A)$ is closed and let $x^{(m)}\in\mathcal I(\ell^\infty(A))$ converge in norm to $x$; put $\varepsilon_m=\|x^{(m)}-x\|$. Then $x\geq0$ and $\|x\|=1$. We show that $x$ satisfies (a) or (b). If (b) holds for infinitely many $m$, then, given $\delta>0$, choose such an $m$ with $1/m+\varepsilon_m<\delta$ and two distinct indices $i\neq j$ with $\|x^{(m)}_i\|,\|x^{(m)}_j\|>1-1/m$; then $\|x_i\|,\|x_j\|>1-\delta$, so $x$ satisfies (b). Otherwise (a) holds for all large $m$, with indices $j_m$. If some index $j$ equals $j_m$ for infinitely many $m$, then $x_j=\lim x^{(m)}_j\in\mathcal I(A)$ because $\mathcal I(A)$ is closed, and (a) holds for $x$. If every index occurs only finitely often among the $j_m$, then for $\delta>0$ pick $m$ with $\varepsilon_m<\delta$ and then $m'$ with $\varepsilon_{m'}<\delta$ and $j_{m'}\neq j_m$; the indices $j_m\neq j_{m'}$ satisfy $\|x_{j_m}\|\geq1-\varepsilon_m>1-\delta$ and likewise for $j_{m'}$, so (b) holds for $x$. In all cases $x\in\mathcal I(\ell^\infty(A))$ by the characterization, which proves the first assertion.

\emph{Second assertion.} Let $a=(a_n)\in\ell_c(A)$ with $a_n=(1-\frac1n)1$. We show that $a\notin\mathcal I(\ell_c(A))$. Suppose $a=\alpha+\beta$ with $\alpha,\beta\in\ell_c(A)_+$, $\alpha\beta=0$, $\|\alpha\|=\|\beta\|=1$. For $n\geq2$ we have $\alpha_n a_n=\alpha_n(\alpha_n+\beta_n)=\alpha_n^2$, so $p_n=\alpha_n/(1-\frac1n)$ is a projection; as $A$ has no nontrivial projection, $\alpha_n\in\{0,a_n\}$ and then $\beta_n=a_n-\alpha_n$. Let $S=\{n\geq2:\alpha_n=a_n\}$ and $T=\{n\geq2:\alpha_n=0\}$. If both $S$ and $T$ were infinite, $\alpha$ would take the values $a_n\to1$ along $S$ and $0$ along $T$ and would not converge (as $1\neq0$ in $A$). Hence $S$ or $T$ is finite. If $S$ is finite, then $\|\alpha\|=\max_{n\in S}(1-\frac1n)<1$ (or $0$), a contradiction; if $T$ is finite, the same argument applied to $\beta$ gives $\|\beta\|<1$, again a contradiction. So $a\notin\mathcal I(\ell_c(A))$.

On the other hand, for $m\geq2$ define $b_m\in\ell_c(A)$ by $b_m(m)=1$ and $b_m(n)=a_n$ for $n\neq m$. Let $\alpha$ be the sequence with $1$ in position $m$ and $0$ elsewhere, and $\beta=b_m-\alpha$, so that $\beta(n)=a_n$ for $n\neq m$ and $\beta(m)=0$. Then $\alpha,\beta\in\ell_c(A)_+$, $\alpha\beta=0$, $\|\alpha\|=1$ and $\|\beta\|=\sup_{n\neq m}(1-\frac1n)=1$, so $b_m\in\mathcal I(\ell_c(A))$. Finally $\|b_m-a\|=\frac1m\to0$. Hence $\mathcal I(\ell_c(A))$ is not norm closed.
\end{proof}

\begin{remark}

The product structure  gives another natural motivation for the study of this property. If $X$ and $Y$ are compact locally path connected spaces, then $X\times Y$ is again compact and locally path connected, and
\[
C(X\times Y)\cong C(X)\otimes_{\min} C(Y).
\]
Consequently, Theorem~\ref{thm:CX} yields the closedness of $\mathcal I(C(X)\otimes_{\min}C(Y))$. This suggests the following question in the noncommutative setting:
\[
\text{If $\mathcal I(A)$ and $\mathcal I(B)$ are closed, is $\mathcal I(A\otimes_{\min}B)$ necessarily closed?}
\]
\end{remark}

\section*{Statements and Declarations}

\subsection*{Competing Interests}
The author declares that there are no competing interests.

\subsection*{Data Availability}
No datasets were generated or analysed during the current study.

\subsection*{Use of AI Tools}
During the preparation of this manuscript, the author used AI-assisted
tools for language editing, LaTeX formatting, and discussion of
mathematical arguments. The author takes full responsibility for the
content, mathematical results, and final version of the manuscript.


\begin{thebibliography}{99}

\bibitem{AP}
C.~A. Akemann and G.~K. Pedersen,
\emph{Ideal perturbation of elements in $C^*$-algebras},
Math. Scand. \textbf{41} (1977), 117--139.

\bibitem{KATO}
T.~Kato,
\emph{Perturbation Theory for Linear Operators},
Springer-Verlag, Berlin, 1980.

\bibitem{Lee2004}
Wha-Suck Lee,
\emph{Ideal perturbation of elements in $C^*$-algebras},
M.Sc. dissertation, University of Pretoria, 2004.

\bibitem{Nagy2013}
G.~Nagy,
\emph{Isometries on positive operators of unit norm},
Publ. Math. Debrecen \textbf{82} (2013), no.~1, 183--192.

\bibitem{Nagy2018}
G.~Nagy,
\emph{Isometries of spaces of normalized positive operators under the
operator norm},
Publ. Math. Debrecen \textbf{92} (2018), no.~1--2, 243--254.

\bibitem{Peralta2018}
A.~M. Peralta,
\emph{Characterizing projections among positive operators in the unit
sphere},
preprint, arXiv:1804.04507.

\bibitem{PeraltaSaavedra2026}
A.~M. Peralta and P.~Saavedra,
\emph{A Metric Characterization of Projections Among Positive Norm-One
Elements in Unital $C^*$-Algebras},
Results Math. \textbf{81} (2026), Article 142.

\bibitem{Taghavi2013}
A.~Taghavi,
\emph{A Note on Positive Zero Divisors in $C^*$-Algebras},
arXiv:1301.3129 (2013).

\end{thebibliography}
\end{document}